\documentclass[a4paper]{amsart}
\usepackage[utf8]{inputenc}
\usepackage[american]{babel}
\usepackage{enumerate}
\usepackage{geometry}
\usepackage{amssymb}
\usepackage[pdfencoding=auto]{hyperref}

\usepackage{csquotes}
\usepackage{amsthm}
\usepackage{amsmath}
\usepackage{lipsum}
\usepackage{xcolor}
\usepackage{orcidlink}
\newcommand{\IR}{\mathbf{R}}
\newcommand{\IN}{\mathbf{N}}
\newcommand{\IH}{\mathbf{H}}
\newcommand{\cH}{\mathcal{H}}
\newcommand{\cF}{\mathcal{F}}
\newtheorem{thm}{Theorem}[section]
\newtheorem{lm}[thm]{Lemma}
\newtheorem{ex}[thm]{Example}
\newtheorem{prop}[thm]{Proposition}

\newtheorem{cor}[thm]{Corollary}

\newtheorem{df}[thm]{Definition}

\newtheorem{theorem}[thm]{Theorem}

\usepackage{enumerate}
\usepackage{mathtools}

\newcommand{\e}{e_{\xi,y}}
\renewcommand{\subset}{\ensuremath{\subseteq}}

\newcommand{\norm}[1]{\left\lVert#1\right\rVert}
\newcommand{\abs}[1]{\lvert#1\rvert}

\DeclarePairedDelimiterX{\inner}[2]{\langle}{\rangle}{#1, #2}

\DeclareMathOperator{\supp}{supp}

\DeclareMathOperator{\dist}{dist}

\DeclareMathOperator{\tr}{tr}

\renewcommand{\Re}{\operatorname{Re}}

\renewcommand{\epsilon}{\ensuremath{\varepsilon}}
\begin{document}
\title{The Coherent State Transform and an Application to the Hyperbolic Laplacian}
\author{Timon Ruben Weinmann \orcidlink{0000-0003-1139-0860}}
\address{Imperial College London}
\email{t.weinmann21@imperial.ac.uk}
\date{\today}

\begin{abstract}
\noindent We review some basic ideas from the theory of overfilling families and use the special case of coherent states to prove Weyl asymptotics for the euclidean Laplacian and the hyperbolic Laplacian on a domain with Dirichlet boundary conditions.

\smallskip


\smallskip

\end{abstract}
\maketitle
\tableofcontents
\section{Introduction}
Suppose that $\Omega\subset\IR^d$ be an open domain of finite measure. Denote by $\lambda_k$, $k\in\IN$ the eigenvalues of the Laplacian on $\Omega$ with Dirichlet boundary conditions.
One version of the Weyl law for the Dirichlet Laplacian reads
$$\lim_{\lambda\to\infty}\lambda^{-1-d/2}\sum_k(\lambda-\lambda_k)_+=\abs{\Omega}\int_{\IR^d}\left(1-\abs{\xi}^2\right)_+\frac{d\xi}{(2\pi)^d}.$$ 
There are countless proofs of this fact, see \cite{book} for instance.
In this paper, we seek to establish a similar asymptotic formula for the hyperbolic Laplacian.
If we consider the operator 
$$H=-\partial_{x_1}^2-e^{2x_1}{\Delta_{\tilde{x}}}$$
where $\Delta_{\tilde{x}}=\sum_{j=2}^d \partial_{x_j}^2$, on a bounded domain in $\Omega\subset\IR^d$ and denote by $\lambda_k$ its Dirichlet eigenvalues, we find
$$\lim_{\lambda\to\infty}\lambda^{-1-d/2}\sum_k (\lambda-{\lambda}_k)_+= \int_{\IR^d\times \Omega}\left(1 - \xi_1^2-e^{2y_1}\abs{\tilde{\xi}}^2\right)_+\frac{d\xi dy}{(2\pi)^d}.$$
The operator $H+\frac{(d-1)^2}{4}$ on $L^2(\IR^d)$ is unitarily equivalent to the hyperbolic Laplacian on $L^2(\IH^d)$. This asymptotic formula thus easily translates to an assertion about the eigenvalues of the hyperbolic Laplacian. In the following we even call $H$ the hyperbolic Laplacian for simplicity.
We prove this asymptotic formula using the theory of coherent states. In section 2 we recall some definitions and theorems from the theory of overfilling families of functions as presented in \cite{Berezin}. This discussion culminates in a proof of the Weyl asymptotic formula for the euclidean Laplacian. This is in fact a special case of the treatment of Weyl asymptotics for fractional Laplacians using coherent states considered in \cite{geisinger}. Section 3 constitutes the principal part of this paper and is dedicated to proving the Weyl asymptotic formula for the hyperbolic Laplacian $H$ using coherent states. Finally, in section 4, we discuss the asymptotic behavior of the error term in the Weyl law for the hyperbolic Laplacian and obtain the more precise formula
$$\sum_k (\lambda-{\lambda}_k)_+=\lambda^{1+d/2} \int_{\IR^{d}\times \Omega}(1- {\xi_1}^2-e^{2y_1}\abs{\tilde{\xi}}^2)_+ \frac{d\xi dy}{(2\pi)^d}+O(\lambda^{d/2+2/3}).$$
Note that this formula also follows from a more general discussion (\cite{bruening}) of the asymptotic behavior of the spectral functions of elliptic differential operators with smooth coefficients. Our approach, however, exemplifies the utility of the coherent state approach, a theory that has classically been applied to powers of euclidean Laplacians and Schr\"odinger operators.
\section{Coherent States}
\subsection{Overfilling Families}
We introduce overfilling families and prove how they can be used to compute traces of operators following \cite{Berezin}.
\begin{df}
Let $\cH$ be a separable Hilbert space and let $(A,\mu)$ be a measure space. A family $e_\alpha\in \cH$, $\alpha\in A$ of elements of $\cH$ indexed by $A$ is said to be overfilling if the map $A\ni \alpha\mapsto e_\alpha\in\cH$ is measurable and if for all $f\in\cH$ one has
$$\norm{f}^2=\int_A \abs{\inner{e_\alpha}{f}}^2 d\mu(\alpha).$$
\end{df}
Given an overfilling family, we obtain an isometric embedding $\cH\to L^2(A)$ by setting
$$f(\alpha)=\inner{e_\alpha}{f}$$
for $f\in\cH$. Note in particular that $e_\alpha(\beta)=\inner{e_\beta}{e_\alpha}=\overline{e_\beta(\alpha)}$. We can therefore use the same notation for the inner products in $\cH$ and in $L^2(A)$. Note that 
$$\inner{g}{f}=\int_A \inner{g}{e_\alpha}\inner{e_\alpha}{f}d\mu(\alpha)$$
for any $g,f\in\cH$.

%
\begin{theorem}\label{trace}
Let $T$ be a non-negative self-adjoint trace-class operator on $\cH$. Then
$$\tr T = \int_A \inner{e_\alpha}{Te_\alpha}d\mu(\alpha)$$
\end{theorem}
\begin{proof}
Indeed let $\psi_k$ be an orthonormal basis of $\cH$ consisting of eigenfunctions of $T$,
\begin{align*}
\int_A \inner{e_\alpha}{Te_\alpha}d\mu(\alpha)&=\int_A\sum_k  \inner{e_\alpha}{\psi_k}\inner{\psi_k}{Te_\alpha}d\mu(\alpha)\\
&=\sum_k {\int_A \inner{e_\alpha}{\psi_k}\inner{\psi_k}{Te_\alpha }d\mu(\alpha)}\\
&=\sum_k {\int_A \inner{T^*\psi_k}{e_\alpha }\inner{e_\alpha}{\psi_k}d\mu(\alpha)}\\
&=\sum_k \inner{\psi_k}{T\psi_k}.
\end{align*}
Note that interchanging the sum and the integral is justified since 
\begin{align*}
\sum_k {\int_A \abs{\inner{e_\alpha}{\psi_k}\inner{\psi_k}{Te_\alpha }}d\mu(\alpha)}&\le \sum_k \left(\int_A \abs{\inner{e_\alpha}{\psi_k}}^2d\mu(\alpha)\right)^{1/2}\left(\int_A\abs{\inner{T\psi_k}{e_\alpha }}^2d\mu(\alpha)\right)^{1/2}\\
&=\sum_k \norm{\psi_k}\norm{T\psi_k}=\sum_k \norm{T\psi_k}<\infty
\end{align*}
by virtue of $T$ being trace-class and the $\psi_k$ being eigenfunctions.
\end{proof}
\subsection{Examples}
Let us discuss two simple examples of overfilling families. 
\begin{ex} Consider the case when $A=\IN$ and $(e_n)_{n\in \IN}$ is an orthonormal basis of $\cH$. In this case the above procedure yields the well-known isometric isomorphism $\cH\to \ell^2(\IN)$ and $f(n)$ is simply the $n$-th coefficient $c_n$ in the Fourier expansions $f=\sum_n c_n e_n$. Given a trace-class operator $T$ on $\cH$ we obtain $\tr T = \sum_k T_{k,k}$, where $T_{k,l}=\inner{e_k}{Te_l}$, as expected.
\end{ex}
\begin{ex}Let $\Omega\subset\IR^d$ be open and bounded. Let $e_\xi(x)=e^{i\xi x}$. By extending $f\in L^2(\Omega)$ by zero outside $\Omega$, we may compute its Fourier transform. Plancherel's theorem yields
$$\int \abs{f(x)}^2dx=\int \abs{\hat{f}(\xi)}^2d\xi=\int \abs{\inner{e_\xi}{f}}^2\frac{d\xi}{(2\pi)^d}.$$
One thus obtains an overfilling family with $A=\IR^d$ and $d\mu(\xi)=(2\pi)^{-d}d\xi$.
\end{ex}
\subsection{Definition of Coherent States}
The rest of this paper will deal with a class of overfilling families in $L^2$-spaces constructed as follows.
Let $g\in C^\infty_c(\IR^d)$ be real-valued and such that $g(-x)=g(x)$ and $\int g(x)^2dx=1$.
For $\xi, y\in \IR^d$ define $\e(x)=e^{i\xi x}g(x-y)$. Clearly $\norm{\e}_{L^2(\IR^d)}=1$. The space $\IR^{2d}$ indexing the family $(\e)$ shall be equipped with the measure $(2\pi)^{-d}d\xi dy$. The functions $\e\in L^2(\IR^d)$ are called coherent states.
\begin{prop}The family $\e$, $(\xi,y)\in\IR^{2d}$ is overfilling in $L^2(\IR^d)$.
\end{prop}
\begin{proof}
We need to show that 
$$\int_{\IR^{2d}}\abs{\inner{\e}{f}}^2\frac{d\xi dy}{(2\pi)^d}=\int_{\IR^d}\abs{f(x)}^2 dx.$$
Note that $\inner{\e}{f}=\int e^{-i\xi x}g(x-y)f(x)dx$ is the Fourier transform of $x\mapsto (2\pi)^{d/2}g(x-y)f(x)$. Therefore, Plancherel's theorem yields
$$\int_{\IR^d}\abs{\inner{\e}{f}}^2 d\xi = (2\pi)^{d}\int_{\IR^d}g(x-y)^2\abs{f(x)}^2dx$$
Integration with respect to $y$ concludes the proof.
\end{proof}
In the following section we will need to make use of the following technical
\begin{lm}\label{symbol1}
$$\inner{\e}{-\partial_{x_j}^2\e}=\xi_j^2+\int_{\IR^d}\abs{\partial_{z_j}g(z)}^2dz$$
\end{lm}
\begin{proof}
Direct computation yields
$$\partial_{x_j}\e(x)=i\xi_j e^{i\xi x}g(x-y)+e^{i\xi x}\partial_{x_j}g(x-y)$$
and 
\begin{equation}\label{tec}
\partial_{x_j}^2\e(x)=-\xi_j^2 e^{i\xi x}g(x-y)+2i\xi_j e^{i\xi x}\partial_{x_j}g(x-y)+e^{i\xi x}\partial_{x_j}^2g(x-y).
\end{equation}
Therefore 
$$\inner{\e}{-\partial_{x_j}^2\e}=\int_{\IR^d}\xi_j^2 (g(x-y))^2dx+2i\xi_j \int_{\IR^d}\partial_{x_j}g(x-y)g(x-y)dx-\int_{\IR^d}\partial_{x_j}^2g(x-y)g(x-y)dx.$$
By symmetry of $g$, the second integral on the right vanishes and by integration by parts, the third integral is equal to
$\int_{\IR^d}(\partial_{z_j}g(z))^2 dz$.
\end{proof}
\begin{lm}The Laplacian acts on the functions $\e$ in the following way:
$$\inner{\e}{-\Delta\e}_{L^2(\IR^d)}=\abs{\xi}^2+\int_\IR \abs{\nabla g(z)}^2dz.$$
\end{lm}
\begin{proof}
Follows immediately from Lemma \ref{symbol1}.
\end{proof}
Instead of the entire space $\IR^d$, consider now an open set $\Omega\subset\IR^d$ of finite measure.
For $\epsilon>0$ let $\Omega_\epsilon=\{x\in \Omega\colon \dist(x,\partial\Omega)>\epsilon\}$. We impose on $g$ the additional condition that $\supp g\subset \{z\in \IR^d\colon \abs{z}<1\}$. Note that $g^\epsilon(z)=\epsilon^{-d/2}g(z/\epsilon)$ still satisfies the symmetry and normalization conditions demanded of $g$. We may therefore define
$\e^\epsilon(x)=e^{i\xi x}g^\epsilon(x-y)$. 

Any function $f\in L^2(\Omega_\epsilon)$ can be extended by zero to a function in $L^2(\IR^d)$. Thus
$$\inner{\e^\epsilon}{f}_{L^2(\IR^d)}=\inner{\e^\epsilon}{f}_{L^2(\Omega_\epsilon)}.$$
The map $y\mapsto \inner{\e^\epsilon}{u}$ vanished outside $\Omega_\epsilon+\{z\in \IR^d\colon \abs{z}<\epsilon\}\subset\Omega$, since $\supp g^\epsilon\subset \{z\in \IR^d\colon \abs{z}<\epsilon\}$. Therefore
$$\int_{\IR^{d}\times \IR^d}\abs{\inner{\e^\epsilon}{f}_{L^2(\IR^d)}}^2\frac{d\xi dy}{(2\pi)^d}=\int_{\IR^{d}\times \Omega}\abs{\inner{\e^\epsilon}{f}_{L^2(\Omega_\epsilon)}}^2\frac{d\xi dy}{(2\pi)^d}$$
Finally, 
$$\int_{\IR^{d}\times \Omega}\abs{\inner{\e^\epsilon}{f}_{L^2(\Omega_\epsilon)}}^2\frac{d\xi dy}{(2\pi)^d}=\norm{f}^2_{L^2(\IR^d)}=\norm{f}^2_{L^2(\Omega_\epsilon)}.$$
This proves the following
\begin{prop}\label{domain}
Let $\Omega\subset\IR^d$ be an open set of finite measure. Assume that $g$ is supported in the open ball $\{z\in\IR^d\colon \abs{z}<1\}$. Then
the family $(\e^\epsilon)$, $y\in\Omega$, $\xi\in\IR^d$ of coherent states constructed using $g^\epsilon$ is overfilling in $L^2(\Omega_\epsilon)$.
\end{prop}

\subsection{Weyl's Law for the Dirichlet Laplacian}
To illustrate the utility of coherent states, we will prove the celebrated Weyl Law for the Dirichlet Laplacian on a domain in euclidean space.

\begin{prop}\label{lapl}
Let $\lambda_k$, $k\in\IN$ be the eigenvalues of the Dirichlet Laplacian in $\Omega$ counted with multiplicity. Then
$$\liminf_{\lambda\to\infty}\lambda^{-1-d/2}\sum_k (\lambda-{\lambda}_k)_+\ge \abs{\Omega} \int_{\IR^d}\left(1 - \abs{\xi}^2\right)_+\frac{d\xi }{(2\pi)^d}.$$
\end{prop}
\begin{proof}
Denote by $\lambda_k$ and $\tilde{\lambda}_k$ the $k$-th Dirichlet eigenvalue of the Laplacian with Dirichlet conditions in $\Omega$, respectively $\Omega_\epsilon$. By domain monotonicity we have $\lambda_k\le \tilde{\lambda}_k$ and consequently $(\lambda-{\lambda}_k)_+\ge (\lambda-\tilde{\lambda}_k)_+$. By Theorem \ref{trace} and Proposition \ref{domain} therefore
$$\sum_k (\lambda-{\lambda}_k)_+\ge \sum_k(\lambda-\tilde{\lambda}_k)_+= \int_{\IR^d\times \Omega}\inner{\e^\epsilon}{(\lambda+\Delta)_+\e^\epsilon}_{L^2(\Omega_\epsilon)}\frac{d\xi dy}{(2\pi)^d}.$$
Clearly
$$\int_{\IR^d\times \Omega}\inner{\e^\epsilon}{(\lambda+\Delta)_+\e^\epsilon}_{L^2(\Omega_\epsilon)}\frac{d\xi dy}{(2\pi)^d}\ge \int_{\IR^d\times \Omega_{2\epsilon}}\inner{\e^\epsilon}{(\lambda+\Delta)_+\e^\epsilon}_{L^2(\Omega_\epsilon)}\frac{d\xi dy}{(2\pi)^d}.$$
For $y\in\Omega_{2\epsilon}$ the function $\e^\epsilon$ is supported in $\Omega_\epsilon$ and consequently
$$\inner{\e^\epsilon}{(\lambda+\Delta)_+\e^\epsilon}_{L^2(\Omega_\epsilon)}=\inner{\e^\epsilon}{(\lambda+\Delta)_+\e^\epsilon}_{L^2(\IR^d)}.$$
Since $\mu\mapsto\mu_+$ is convex, applying Jensen's inequality to the spectral measure of $-\Delta$, we obtain
$$\inner{\e^\epsilon}{(\lambda+\Delta)_+\e^\epsilon}_{L^2(\IR^d)}\ge \left(\inner{\e^\epsilon}{(\lambda+\Delta)\e^\epsilon}_{L^2(\IR^d)}\right)_+= \left(\lambda -\abs{\xi}^2-\int_{\IR^d} \abs{\nabla g^\epsilon(z)}^2dz\right)_+.$$
Consequently,
\begin{align}\label{lowerE}
\sum_k(\lambda-\lambda_k)_+&\ge \int_{\IR^d\times\Omega_{2\epsilon}}\left(\lambda -\abs{\xi}^2-\int_{\IR^d} \abs{\nabla g^\epsilon(z)}^2dz\right)_+\frac{d\xi dy}{(2\pi)^d}\nonumber\\\nonumber
&=\lambda^{1+d/2}\int_{\IR^d\times\Omega_{2\epsilon}}\left(1 -\abs{\xi}^2-\frac{1}{\lambda}\int_{\IR^d} \abs{\nabla g^\epsilon(z)}^2dz\right)_+\frac{d\xi dy}{(2\pi)^d}\\
&\ge\lambda^{1+d/2}\abs{\Omega_{2\epsilon}}\int_{\IR^d}\left(1 -\abs{\xi}^2-\frac{1}{\lambda}\int_{\IR^d} \abs{\nabla g^\epsilon(z)}^2dz\right)_+\frac{d\xi}{(2\pi)^d}.
\end{align}
\color{black}
Finally we obtain
$$\liminf_{\lambda\to\infty}\lambda^{-1-d/2}\sum_k(\lambda-\lambda_k)_+\ge \abs{\Omega_{2\epsilon}}\int_{\IR^d}\left(1 -\abs{\xi}^2\right)_+\frac{d\xi}{(2\pi)^d}$$
by dominated convergence. Letting $\epsilon\to 0$ concludes the proof, since $1_{\Omega_{2\epsilon}}(x)\to1_{\Omega}(x)$ for all $x\in\IR^d$ and thus $\abs{\Omega_{2\epsilon}}\to \abs{\Omega}$ by dominated convergence.
\end{proof}
We now seek to prove an upper bound on $\sum_k(\lambda-\lambda_k)_+$. This result is due to Li and Yau \cite{LYa} and our proof follows \cite{book}.
\begin{prop}\label{upperE}
$$\sum_k(\lambda-\lambda_k)_+\le \abs{\Omega}\int_{\IR^d}\left(\lambda-\abs{\xi}^2\right)_+\frac{d\xi}{(2\pi)^d}$$
\end{prop}
\begin{proof}
Let $(\psi_k)_k\subset L^2(\Omega)$ be an orthonormal system of eigenfunctions corresponding to $(\lambda_k)_k$. In particular, $\psi_k\in H^1_0(\Omega)$. Then, since every $\psi_k$ can be extended by zero to a function $H^1(\IR^d)$, we have
$$\lambda_k=\int_{\Omega}\abs{\nabla \psi_k(x)}^2dx=\int_{\IR^d}\abs{\xi}^2\abs{\hat{\psi}_k(\xi)}^2d\xi.$$
Therefore
$$(\lambda-\lambda_k)_+=\left(\int_{\IR^d}(\lambda-\abs{\xi}^2)\abs{\hat{\psi}_k(\xi)}^2d\xi\right)_+\le \int_{\IR^d}\left(\lambda-\abs{\xi}^2\right)_+\abs{\hat{\psi}_k(\xi)}^2d\xi,$$
by Jensen's inequality, since $\abs{\hat{\psi}_k(\xi)}^2d\xi$ is a probability measure.
\begin{align*}
\sum_k (\lambda-\lambda_k)_+ &\le \sum_{\lambda_k<\lambda} \int_{\IR^d}\left(\lambda-\abs{\xi}^2\right)_+\abs{\hat{\psi}_k(\xi)}^2d\xi\\
&= \int_{\IR^d}\left(\lambda-\abs{\xi}^2\right)_+\sum_{\lambda_k<\lambda}\abs{\hat{\psi}_k(\xi)}^2d\xi\\
&\le \int_{\IR^d}\left(\lambda-\abs{\xi}^2\right)_+\sum_{k}\abs{\hat{\psi}_k(\xi)}^2d\xi.
\end{align*}
Note that 
$$\hat{\psi}_k(\xi)=\frac{1}{(2\pi)^{d/2}}\int_{\IR^d}e^{-ix\xi}\psi_k(x)dx=\frac{1}{(2\pi)^{d/2}}\int_{\Omega}e^{-i\xi x}\psi_k(x)dx=\inner{(2\pi)^{-d/2}e^{i\xi \cdot}}{\psi_k}_{L^2(\Omega)}.$$
Since the $\psi_k$ form an orthonormal basis of $L^2(\Omega)$, we obtain 
$$\sum_{k}\abs{\hat{\psi}_k(\xi)}^2
=\sum_{k}\abs{\inner{(2\pi)^{-d/2}e^{i\xi \cdot}}{\psi_k}_{L^2(\Omega)}}^2
=\norm{(2\pi)^{-d/2}e^{i\xi \cdot}}_{L^2(\Omega)}^2=(2\pi)^{-d}\abs{\Omega}.$$
This concludes the argument.
\end{proof}
\begin{cor}For the eigenvalues $\lambda_k$ of the Dirichlet Laplacian on a domain $\Omega$ of finite measure, the following asymptotic formula holds
$$\lim_{\lambda\to\infty}\lambda^{-1-d/2}\sum_k(\lambda-\lambda_k)_+=\abs{\Omega}\int_{\IR^d}\left(1-\abs{\xi}^2\right)_+\frac{d\xi}{(2\pi)^d}.$$
\end{cor}

\section{The Hyperbolic Laplacian}
\subsection{Preliminaries}
On a Riemannian manifold $(M,g)$, the Laplace--Beltrami operator is locally given by
$$-\Delta_g=-\frac{1}{\sqrt{g}}\partial_{x^i}\left(\sqrt{g}g^{ij}\partial_{x^j}\right),$$
whereas the canonical Riemannian volume is given by
$$dV_g=\sqrt{g}dx^1 dx^2 ...\ dx^d.$$

Consider now the hyperbolic space in the half-plane model
$$M=\IH^d=\{(y,x)\in\IR\times\IR^{d-1}\colon\ y>0\},\ g=\frac{1}{y^2}(dy^2+dx^2).$$
The Laplace--Beltrami operator in this case is
$$-\Delta_g=-y^d\partial_y(y^{2-d}\partial_y)-y^2\Delta_x,$$
where $\Delta_x=\sum_{j=1}^{d-1}\partial_{x_j}^2$.
Performing the change of coordinates $y=e^t$, one obtains a unitary operator $L^2(\IH^d, y^{-d}dx\ dy)\to L^2(\IR^d, dx\ dt)$ and $-\Delta_g$ acts on $L^2(\IR^d, dx\ dt)$ as $-\partial_t^2-e^{2t}\Delta_x+(d-1)^2/4$. This unitary equivalence between $-\Delta_g$ and $-\partial_t^2-e^{2t}\Delta_x+(d-1)^2/4$ causes their eigenvalues to coincide.
In \cite{hyper} this correspondence was used to extend a class of Lieb--Thirring inequalities from $\IR^d$ to $\IH^d$. In the rest of this section we adopt the notation $\tilde{x}=(x_2,...,x_d)$ for $x\in\IR^d$.
\subsection{Lower Trace Bound}
Consider the operator 
$$H=-\partial_{x_1}^2-e^{2x_1}{\Delta_{\tilde{x}}}$$
where $\Delta_{\tilde{x}}=\sum_{j=2}^d \partial_{x_j}^2$ on the space $L^2(\IR^d)$. We are interested in its action on a coherent state of the form $\e(x)=e^{i\xi x}g(x-y)$.
\begin{lm}\label{symbol}
$$\inner{\e}{H\e}_{L^2(\IR^d)}=\xi_1^2+e^{2y_1}\abs{\tilde{\xi}}^2\int_{\IR^d}e^{2z_1}g(z)^2dz+e^{2y_1}\int_{\IR^d} e^{2z_1}\abs{{\nabla_{\tilde{z}}}g(z)}^2dz+\int_{\IR^d}(\partial_{z_1}g(z))^2 dz$$
\end{lm}
\begin{proof}
The contribution of $-\partial^2_{x_1}$ was already determined in \ref{symbol1}.
We compute $\inner{\e}{-e^{2x_1}\partial_{x_j}^2\e}_{L^2(\IR^d)}$ for $j>1$. Equation \ref{tec} in the proof of Lemma \ref{symbol1} implies
\begin{align*}
\inner{\e}{-e^{2x_1}\partial_{x_j}^2\e}=&\xi_j^2\int_{\IR^d}e^{2x_1}(g(x-y))^2dx
+2i\xi_j \int_{\IR^d}e^{2x_1}\partial_{x_j}g(x-y)g(x-y)dx\\
&-\int_{\IR^d}e^{2x_1}\partial_{x_j}^2g(x-y)g(x-y)dx.\\
\end{align*}
The first term on the right equals
$$e^{2y_1}\xi_j^2\int_{\IR^d}e^{2(x_1-y_1)}(g(x-y))^2dx=e^{2y_1}\xi_j^2\int_{\IR^d}e^{2z_1}(g(z))^2dz.$$
For the second integral we find
$$\int_{\IR^d}e^{2x_1}\partial_{x_j}g(x-y)g(x-y)dx=\int_{\IR}e^{2x_1}\int_{\IR^{d-1}}\partial_{x_j}g(x-y)g(x-y) d\tilde{x} dx_1,$$
where the inner integral vanishes due to symmetry of $g$. Finally, for the third integral we find, again by splitting the integral, that
\begin{align*}
\int_{\IR^d}e^{2x_1}\partial_{x_j}^2g(x-y)g(x-y)dx
&=-e^{2y_1}\int_{\IR}e^{2(x_1-y_1)}\int_{\IR^{d-1}}\partial_{x_j}^2g(x-y)g(x-y)d\tilde{x}dx_1\\
&=e^{2y_1}\int_{\IR}e^{2(x_1-y_1)}\int_{\IR^{d-1}}(\partial_{x_j}g(x-y))^2d\tilde{x}dx_1\\
&=e^{2y_1}\int_{\IR^{d}}e^{2z_1}(\partial_{z_j}g(z))^2dz
\end{align*}
Summing over $j=2,...,d$ completes the argument.
\end{proof}

\begin{theorem}Let $\Omega\subset\IR^d$ be a bounded open set and let $\lambda_k$, $k\in\IN$ be the eigenvalues of the operator $H$ on $\Omega$ with Dirichlet boundary conditions. Then
$$\liminf_{\lambda\to\infty}\lambda^{-1-d/2}\sum_k (\lambda-{\lambda}_k)_+\ge \int_{\IR^d\times \Omega}\left(1 - \xi_1^2-e^{2y_1}\abs{\tilde{\xi}}^2\right)_+\frac{d\xi dy}{(2\pi)^d}.$$
\end{theorem}

\begin{proof}

%
%
Let us again consider a $g$ supported in the unit ball and construct $\e^\epsilon$ accordingly.
By the same argument as for the Dirichlet Laplacian in the proof of Proposition \ref{lapl}, we have
$$\sum_k (\lambda-{\lambda}_k)_+\ge \int_{\IR^d\times \Omega_{2\epsilon}}\left(\inner{\e^\epsilon}{(\lambda-H)\e^\epsilon}_{L^2(\IR^d)}\right)_+\frac{d\xi dy}{(2\pi)^d}.$$
Using the expression obtained for the symbol of $H$ in \ref{symbol} then gives
\begin{equation}\label{low}
\sum_k (\lambda-{\lambda}_k)_+\ge \int_{\IR^d\times {\Omega_{2\epsilon}}}\left(\lambda - \xi_1^2-e^{2y_1}\abs{\tilde{\xi}}^2c^\epsilon_3-e^{2y_1}c^\epsilon_2-c^\epsilon_1 \right)_+\frac{d\xi dy}{(2\pi)^d},
\end{equation}
where
\begin{align*}c^\epsilon_1&=\int_{\IR^d}(\partial_{z_1}g^\epsilon(z))^2 dz\\
c^\epsilon_2&=\int_{\IR^d}e^{2z_1}\abs{{\nabla_{\tilde{z}}}g^\epsilon(z)}^2dz\\
c^\epsilon_3&=\int_{\IR^d}e^{2z_1}g^\epsilon(z)^2dz.
\end{align*}
Note that since $g(z)^2$ is a mollifier, we have $c_3^\epsilon\to 1$ as $\epsilon\to 0$.
After a change of coordinates, we obtain
\begin{equation}\label{lowerH}
\sum_k (\lambda-{\lambda}_k)_+\ge\lambda^{1+d/2}\int_{\IR^d\times \Omega_{2\epsilon}}\left(1- \xi_1^2-e^{2y_1}\abs{\tilde{\xi}}^2c^\epsilon_3-\frac{1}{\lambda}e^{2y_1}c^\epsilon_2-\frac{1}{\lambda}c^\epsilon_1 \right)_+\frac{d\xi dy}{(2\pi)^d}.
\end{equation}

Dominated convergence yields
$$\liminf_{\lambda\to\infty}\lambda^{-1-d/2}\sum_k (\lambda-{\lambda}_k)_+\ge \int_{\IR^d\times \Omega_{2\epsilon}}\left(1- \xi_1^2-e^{2y_1}\abs{\tilde{\xi}}^2c^\epsilon_3\right)_+\frac{d\xi dy}{(2\pi)^d}.$$
Letting $\epsilon\to 0$ and using dominated convergence concludes the argument.
\end{proof}

\subsection{Upper Trace Bound}
Before proving an upper bound on the trace of the hyperbolic Laplacian, we need to consider some further technicalities concerning coherent states.
Consider again the general situation where $g\in C^\infty_c(\IR^d)$ is real-valued and such that $g(-x)=g(x)$ and $\int g(x)^2dx=1$.
As before, define $\e(x)=e^{i\xi x}g(x-y)$.
The operator $\Phi\colon L^2(\IR^d, dx)\to L^2(\IR^{2d}, (2\pi)^{-d}dyd\xi)$ given by $(\Phi \psi)(y,\xi)=\inner{\e}{\psi}_{L^2(\IR^d)}$ is called the coherent state transform. Since $(\e)$ is overfilling, $\Phi$ is an isometry, i.e. $\Phi^*\Phi=I_{L^2(\IR^d)}$. 


In this section's theorem we shall make use of the two following lemmas:
\begin{lm}
Let $\psi\in H^1(\IR^d)$. Then
$$\inner{\psi}{\Phi^* \xi_1^2\Phi\psi}=\norm{\partial_{x_1}\psi}^2+\int_{\IR^d}(\partial_{z_1}g(z))^2 dz \norm{\psi}^2.$$
\end{lm}
\begin{proof}
Let us introduce some notation. We write $g_y(x)=g(x-y)$. The Fourier transform is denoted by $\cF\colon L^2(\IR^d)\to L^2(\IR^d)$ with the following convention for $f\in L^2(\IR^d)\cap L^1(\IR^d)$:
$$\cF f (\xi) = (2\pi)^{-d/2}\int_{\IR^d}e^{-ix\xi}f(x)dx.$$
Finally, we write $X=L^2(\IR^d, dx)$ and $Y=L^2(\IR^{2d}, (2\pi)^{-d}dyd\xi)$ for brevity.

The coherent state transform of $\psi$ is given by
$$\Phi\psi(y,\xi)=\inner{\e}{\psi}_X=\int_{\IR^d}e^{-ix\xi}g_y(x)f(x)dx=(2\pi)^{d/2}\cF(g_y f).$$
Consequently,
\begin{align*}
\inner{\psi}{\Phi^* \xi_1^2\Phi\psi}_X&=\inner{\xi_1\Phi\psi}{ \xi_1\Phi\psi}_Y\\
&=(2\pi)^{d}\norm{\xi_1\cF(g_y \psi)}^2_Y\\
&=(2\pi)^d\norm{\cF(\partial_{x_1}(g_y \psi))}^2_Y\\
&=\int_{\IR^{2d}} \abs{\cF(\partial_{x_1}(g_y \psi))(\xi)}^2 d\xi dy\\
&=\int_{\IR^{d}}\int_{\IR^{d}} \abs{\partial_{x_1}(g_y \psi)(x)}^2 dx dy\\
&=\int_{\IR^{d}}\int_{\IR^{d}}\abs{\partial_{x_1}(g_y \psi)(x)}^2 dy dx
\end{align*}
For the integrand we find
$$\abs{\partial_{x_1}(g(x-y){\psi(x)})}^2=\abs{\partial_{x_1}g(x-y){\psi(x)}}^2+\abs{g(x-y)\partial_{x_1}{\psi(x)}}^2+2\Re \overline{{\partial_{x_1}g(x-y){\psi(x)}}}g(x-y)\partial_{x_1}{\psi(x)}.$$
Since $g(-z)=g(z)$, it is $(\partial_{z_1}g)(-z)=-(\partial_{z_1}g)(z)$ and thus
$$\int_{\IR^d}{{\partial_{x_1}g(x-y)}}g(x-y)dy=\int_{\IR^d}{{\partial_{z_1}g(z)}}g(z)dz=0.$$
Finally,
\begin{align*}
\inner{\psi}{\Phi^* \xi_1^2\Phi\psi}&=\int_{\IR^{2d}}\left( \abs{\partial_{x_1}g(x-y){\psi(x)}}^2+\abs{g(x-y)\partial_{x_1}{\psi(x)}}^2\right){dy dx}\\
&=\int_{\IR^{d}} \left(\int_{\IR^d}\abs{\partial_{x_1}g(x-y)}^2dy\right)\abs{ \psi(x)}^2dx+\int_{\IR^{d}} \left(\int_{\IR^{d}}  \abs{g(x-y)}^2dy\right)\abs{\partial_{x_1}\psi(x)}^2 dx\\
&=\int_{\IR^d}\abs{\partial_{z_1}g(z)}^2dz\int_{\IR^{d}} \abs{ \psi(x)}^2dx+\int_{\IR^{d}}  \abs{g(z)}^2dz\int_{\IR^{d}}\abs{\partial_{x_1}\psi(x)}^2 dx\\
&=\int_{\IR^{d}}(\partial_{z_1}g(z))^2 dz\norm{\psi}^2+\norm{\partial_{x_1}{\psi}}^2.\qedhere
\end{align*}
\end{proof}
\begin{lm}Assume that the function in the definition of coherent states can be decomposed as $g(z)=g_1(z_1)\tilde{g}(\tilde{z})$ where
$g_1$ and $\tilde{g}$ are symmetric in the sense that $g_1(-x_1)=g_1(x_1)$ and $\tilde{g}(-\tilde{x})=\tilde{g}(\tilde{x})$, compactly supported, and such that
$\int_\IR g_1(z_1)^2 dz_1=\int_{\IR^{d-1}} \tilde{g}(\tilde{z})^2 d\tilde{z}=1$.
Let $\psi\in H^1(\IR^d)$ have compact support. Then
$$\inner{\psi}{\Phi^* e^{2y_1}\abs{\tilde{\xi}}^2\Phi\psi}=\int_{\IR^d}e^{2z_1}\abs{{\nabla_{\tilde{z}}}g(z)}^2dz\norm{e^{x_1}\psi}^2+\int_{\IR^d}e^{2z_1}g(z)^2dz\norm{e^{x_1}\nabla_{\tilde{x}}\psi}^2.$$
\end{lm}
\begin{proof}
We proceed in a similar way as in the proof of the lemma above. Adopting the same notation as before, we find for $j=2,...,d$ that
\begin{align*}
    \inner{\psi}{\Phi^* e^{2y_1}\xi_j^2\Phi\psi}_X&=\inner{e^{y_1}\xi_1\Phi\psi}{ e^{y_1}\xi_j\Phi\psi}_Y\\
&=(2\pi)^{d}\norm{e^{y_1}\xi_j\cF(g_y \psi)}^2_Y\\
&=(2\pi)^d\norm{e^{y_1}\cF(\partial_{x_j}(g_y \psi))}^2_Y\\
&=\int_{\IR^{2d}} e^{2y_1}\abs{\cF(\partial_{x_j}(g_y \psi))(\xi)}^2 d\xi dy\\
&=\int_{\IR^{d}} e^{2y_1} \int_{\IR^{d}}\abs{\cF(\partial_{x_j}(g_y \psi))(\xi)}^2 d\xi\ dy\\
&=\int_{\IR^{d}} e^{2y_1} \int_{\IR^{d}}\abs{(\partial_{x_j}(g_y \psi))(\xi)}^2 dx\ dy\\
&=\int_{\IR^{d}}\int_{\IR^{d}}e^{2y_1} \abs{\partial_{x_j}(g_y \psi)(x)}^2 dy dx.
\end{align*}
Summing over $j$ we obtain
\begin{align*}
\inner{\psi}{\Phi^* e^{2y_1}\abs{\tilde{\xi}}^2\Phi\psi}
&=\int_{\IR^{2d}}e^{2y_1}\abs{\nabla_{\tilde{x}}(g(x-y){\psi(x)})}^2{dy dx}.
\end{align*}
Note that
\begin{align*}
\abs{\nabla_{\tilde{x}}(g(x-y){\psi(x)})}^2&=\abs{\nabla_{\tilde{x}}g(x-y)}^2\abs{\psi(x)}^2+\abs{g(x-y)}^2\abs{\nabla_{\tilde{x}}\psi(x)}^2\\
&+\sum_{j=2}^d 2\Re \overline{{\partial_{x_j}g(x-y){\psi(x)}}}g(x-y)\partial_{x_j}{\psi(x)},
\end{align*}
and that
$$\int_{\IR^d}e^{2y_1}{{\partial_{x_j}g(x-y)}}g(x-y)dy=\int_{\IR}e^{2y_1} g_1(x_1-y_1)^2dy_1\int_{\IR^{d-1}}{{\partial_{x_j}\tilde{g}(\tilde{x}-\tilde{y})}}\tilde{g}(\tilde{x}-\tilde{y})d\tilde{y},$$
where the second integral on the right vanishes due to the symmetry of $\tilde{g}$. Finally
\begin{align*}
\int_{\IR^{2d}}e^{2y_1}\abs{\nabla_{\tilde{x}}g(x-y)}^2\abs{\psi(x)}^2dxdy&=\int_{\IR^{2d}}e^{2(y_1-x)}\abs{\nabla_{\tilde{x}}g(x-y)}^2e^{2x_1}\abs{\psi(x)}^2dxdy\\
&=\int_{\IR^{d}}e^{2z}\abs{\nabla_{\tilde{z}}g(z)}^2dz\int_{\IR^{d}}e^{2x_1}\abs{\psi(x)}^2dx,
\end{align*}
and
\begin{align*}
\int_{\IR^{2d}}e^{2y_1}\abs{g(x-y)}^2\abs{\nabla_{\tilde{x}}\psi(x)}^2dxdy&=\int_{\IR^{2d}}e^{2(y_1-x_1)}\abs{g(x-y)}^2e^{2x_1}\abs{\nabla_{\tilde{x}}\psi(x)}^2dxdy\\
&=\int_{\IR^{d}}e^{2z_1}\abs{g(z)}^2dz\int_{\IR^{d}}e^{2x_1}\abs{\nabla_{\tilde{x}}\psi(x)}^2dx.
\end{align*}

This concludes the proof.
\end{proof}
We are now ready to state the main theorem of this subsection.
\begin{theorem}Let $\Omega\subset\IR^d$ be a bounded open set and let $\lambda_k$, $k\in\IN$ be the eigenvalues of the operator $H$ on $\Omega$ with Dirichlet boundary conditions. Then
$$\limsup_{\lambda\to\infty}\lambda^{-1-d/2}\sum_k (\lambda-{\lambda}_k)_+\le \int_{\IR^d\times \Omega}\left(1 - \xi_1^2-e^{2y_1}\abs{\tilde{\xi}}^2\right)_+\frac{d\xi dy}{(2\pi)^d}$$
\end{theorem}
\begin{proof}
Assume that $g$ is supported in the unit ball $B_1(0)$ and that it can be decomposed as $g(z)=g_1(z_1)\tilde{g}(\tilde{z})$, where
$g_1,\tilde{g}\ge 0$ are symmetric in the sense that $g_1(-x_1)=g_1(x_1)$ and $\tilde{g}(-\tilde{x})=\tilde{g}(\tilde{x})$, compactly supported and satisfy
$\int_\IR g_1(z_1)^2 dz_1=\int_{\IR^{d-1}} \tilde{g}(\tilde{z})^2 d\tilde{z}=1$. Define $g^\epsilon(z)=\epsilon^{-d/2}g(z/\epsilon)$.  Denote by $\Omega^\epsilon$ the set $\{x\in\IR^d\colon \dist(x,\Omega)<\epsilon\}$.
As before, let $\lambda_k$, $k\in\IN$ be the eigenvalues of $H$. Since $H$ is clearly elliptic, we can find according eigenfunctions $\psi_k\in H^1_0(\Omega)$ such that $(\psi_k)_k$ is an orthonormal basis of $L^2(\Omega)$ (see \cite{evans}). Extending these eigenfunctions by zero outside $\Omega$, we obtain functions in $H^1(\IR^d)$ that are compactly supported. We can therefore apply the preceding two lemmas to the functions $\psi_k$.
For the eigenvalues we find 
$$\lambda_k=\inner{\psi_k}{H\psi_k}_{L^2(\Omega)}=\inner{\psi_k}{H\psi_k}_{L^2(\IR^d)}=\norm{\partial_{x_1}\psi_k}^2+\norm{e^{x_1}\nabla_{\tilde{x}}\psi_k}^2.$$
For brevity we denote, again, 
\begin{align*}c^\epsilon_1&=\int_{\IR^d}(\partial_{z_1}g^\epsilon(z))^2 dz\\
c^\epsilon_2&=\int_{\IR^d}e^{2z_1}\abs{{\nabla_{\tilde{z}}}g^\epsilon(z)}^2dz\\
c^\epsilon_3&=\int_{\IR^d}e^{2z_1}g^\epsilon(z)^2dz.
\end{align*}
Denoting $R=\sup\{y_1|\ y\in\Omega\}$, we find
\begin{align*}
(\lambda-\lambda_k)_+&=\left(\lambda-\inner{\psi_k}{\Phi^* {\xi_1}^2\Phi\psi_k}+c^\epsilon_1-(c^\epsilon_3)^{-1}\inner{\psi_k}{\Phi^* e^{2y_1}\abs{\tilde{\xi}}^2\Phi\psi_k}+(c^\epsilon_3)^{-1}c^\epsilon_2\norm{e^{x_1}\psi_k}^2\right)_+\\
&\le\left(\lambda-\inner{\psi_k}{\Phi^* {\xi_1}^2\Phi\psi_k}+c^\epsilon_1-(c^\epsilon_3)^{-1}\inner{\psi_k}{\Phi^* e^{2y_1}\abs{\tilde{\xi}}^2\Phi\psi_k}+(c^\epsilon_3)^{-1}c^\epsilon_2e^{2R}\right)_+\\
&=\left(\inner{\Phi\psi_k}{(\lambda-\xi_1^2-(c^\epsilon_3)^{-1}e^{2y_1}\abs{\tilde{\xi}}^2+c^\epsilon_1+(c^\epsilon_3)^{-1}c^\epsilon_2e^{2R})\Phi\psi_k}\right)_+\\
&\le\inner{\Phi\psi_k}{(\lambda-\xi_1^2-(c^\epsilon_3)^{-1}e^{2y_1}\abs{\tilde{\xi}}^2+c^\epsilon_1+(c^\epsilon_3)^{-1}c^\epsilon_2e^{2R})_+\Phi\psi_k}
\end{align*}
where the last inequality is due to Jensen's inequality. Summing over $k$ then yields
$$\sum_k(\lambda-\lambda_k)_+\le \int_{\IR^{2d}}(\lambda-\xi_1^2-(c^\epsilon_3)^{-1}e^{2y_1}\abs{\tilde{\xi}}^2+c^\epsilon_1+(c^\epsilon_3)^{-1}c^\epsilon_2e^{2R})_+\sum_k\abs{\Phi\psi_k(\xi,y)}^2 \frac{d\xi dy}{(2\pi)^d}.$$
Since $(\psi_k)$ is an orthonormal basis, we obtain
$$\sum_k\abs{\Phi\psi_k(\xi,y)}^2=\sum_k\abs{\inner{\e}{\psi_k}}^2=\norm{\e}^2_{L^2(\Omega)}\le \norm{\e}^2_{L^2(\IR^d)}=1$$
Since the map $y\mapsto \Phi\psi (y,\xi)$ is supported in $\supp\psi + \supp g^\epsilon\subset \overline{\Omega} + B_\epsilon(0)\subset\Omega^\epsilon$ for $\psi\in H_0^1(\Omega)$, we find
$$\sum_k\abs{\Phi\psi_k(\xi,y)}^2\le \chi_{\Omega^\epsilon}(y).$$
Hence
\begin{align*}
\sum_k(\lambda-\lambda_k)_+&\le \int_{\IR^{d}\times \Omega^\epsilon}(\lambda-\xi_1^2-(c^\epsilon_3)^{-1}e^{2y_1}\abs{\tilde{\xi}}^2+c^\epsilon_1+(c^\epsilon_3)^{-1}c^\epsilon_2e^{2R})_+ \frac{d\xi dy}{(2\pi)^d}\\
&=\lambda^{1+d/2}\int_{\IR^{d}\times \Omega^\epsilon}(1-\xi_1^2-(c^\epsilon_3)^{-1}e^{2y_1}\abs{\tilde{\xi}}^2+c^\epsilon_1\lambda^{-1}+(c^\epsilon_3)^{-1}c^\epsilon_2e^{2R}\lambda^{-1})_+ \frac{d\xi dy}{(2\pi)^d}.
\end{align*}
Consequently 
\begin{align*}
\limsup_{\lambda\to\infty}\lambda^{-1-d/2}\sum_k(\lambda-\lambda_k)_+&\le \int_{\IR^{d}\times \Omega^\epsilon}(1-\xi_1^2-(c^\epsilon_3)^{-1}e^{2y_1}\abs{\tilde{\xi}}^2)_+ \frac{d\xi dy}{(2\pi)^d}.
\end{align*}
Letting $\epsilon\to 0$ finishes the proof by dominated convergence and since $c_3^\epsilon\to 1$.
\end{proof}
\begin{cor}
Let $\Omega\subset\IR^d$ be a bounded open set and let $\lambda_k$, $k\in\IN$ be the eigenvalues of the operator $H$ on $\Omega$ with Dirichlet boundary conditions. Then
$$\lim_{\lambda\to\infty}\lambda^{-1-d/2}\sum_k (\lambda-{\lambda}_k)_+= \int_{\IR^d\times \Omega}\left(1 - \xi_1^2-e^{2y_1}\abs{\tilde{\xi}}^2\right)_+\frac{d\xi dy}{(2\pi)^d}$$
\end{cor}

\section{Asymptotic Formula}
Recall that for the euclidean Laplacian with Dirichlet boundary conditions we obtained inequality \ref{lowerE} reading:
$$\sum_k(\lambda-\lambda_k)_+\ge\lambda^{1+d/2}\abs{\Omega_{2\epsilon}}\int_{\IR^d}\left(1 -\abs{\xi}^2-\frac{1}{\lambda}\int_{\IR^d} \abs{\nabla g^\epsilon(z)}^2dz\right)_+\frac{d\xi}{(2\pi)^d}.$$
Denoting $B=\{\xi\in\IR^d|\ \abs{\xi}^2\le 1\}$, we can clearly make the estimate
\begin{align*}
\sum_k(\lambda-\lambda_k)_+&\ge\lambda^{1+d/2}\abs{\Omega_{2\epsilon}}\int_{B}\left(1 -\abs{\xi}^2-\frac{1}{\lambda}\int_{\IR^d} \abs{\nabla g^\epsilon(z)}^2dz\right)_+\frac{d\xi}{(2\pi)^d}\\
&\ge \lambda^{1+d/2}\abs{\Omega_{2\epsilon}}\int_{B}\left(1 -\abs{\xi}^2\right)_+\frac{d\xi}{(2\pi)^d}-\lambda^{d/2}\frac{\abs{\Omega_{2\epsilon}}\abs{B}}{(2\pi)^d}\int_{\IR^d} \abs{\nabla g^\epsilon(z)}^2dz\\
&\ge \lambda^{1+d/2}\abs{\Omega_{2\epsilon}}\int_{B}\left(1 -\abs{\xi}^2\right)_+\frac{d\xi}{(2\pi)^d}-\lambda^{d/2}\frac{\abs{\Omega}\abs{B}}{(2\pi)^d}\int_{\IR^d} \abs{\nabla g^\epsilon(z)}^2dz\\
&\ge \lambda^{1+d/2}\abs{\Omega}\int_{B}\left(1 -\abs{\xi}^2\right)_+\frac{d\xi}{(2\pi)^d}-R(\lambda,\epsilon)
\end{align*}
where
$$R(\lambda,\epsilon)=\lambda^{1+d/2}\left(\abs{\Omega}-\abs{\Omega_{2\epsilon}}\right)\int_{B}\left(1 -\abs{\xi}^2\right)_+\frac{d\xi}{(2\pi)^d}+\lambda^{d/2}\frac{\abs{\Omega}\abs{B}}{(2\pi)^d}\int_{\IR^d} \abs{\nabla g^\epsilon(z)}^2dz$$

Note that $$\int_{\IR^d} \abs{\nabla g^\epsilon(z)}^2dz=\epsilon^{-2}\int_{\IR^d} \abs{\nabla g(z)}^2dz.$$
Let $\epsilon=\lambda^{-\alpha}$ with $\alpha>0$. Then there are constants $L,C\ge 0$ such that
$$R(\lambda,\epsilon)=L \lambda^{1+d/2}\left(\abs{\Omega}-\abs{\Omega_{2\epsilon}}\right) + C \lambda^{d/2+2\alpha}.$$
Suppose that $\abs{\Omega}-\abs{\Omega_{2\epsilon}}=O(\epsilon)$. This is in particular the case when the inner Minkowski content of $\Omega$ exists which was proven to be the case when the boundary of $\Omega$ is Lipschitz-regular and $\Omega$ itself is bounded in \cite{ambrosio}.
Then,
$$R(\lambda, \lambda^{-\alpha})= C \lambda^{d/2+2\alpha}+ O(\lambda^{1+d/2-\alpha}).$$
Minimizing $\max\{d/2+2\alpha,\ 1+d/2-\alpha\}$, we find $\alpha=1/3$. Consequently
$$\sum_k(\lambda-\lambda_k)_+\ge \lambda^{1+d/2}\abs{\Omega}\int_{B}\left(1 -\abs{\xi}^2\right)_+\frac{d\xi}{(2\pi)^d}+O(\lambda^{d/2+2/3}).$$
Combining this with the upper estimate \ref{upperE} we obtain the following 
\begin{theorem}
Let $\Omega$ be a bounded domain whose boundary is Lipschitz-regular. Then the eigenvalues $\lambda_k$ of the Dirichlet Laplacian on $\Omega$ satisfy
$$\sum_k(\lambda-\lambda_k)_+= \lambda^{1+d/2}\abs{\Omega}\int_{B}\left(1 -\abs{\xi}^2\right)_+\frac{d\xi}{(2\pi)^d}+O(\lambda^{d/2+2/3}).$$
\end{theorem}
This is a significantly weaker result than the one obtained in \cite{larson} where the following two-term asymptotic formula was proven for bounded Lipschitz domains:
$$\sum_k(\lambda-\lambda_k)_+=L_d\abs{\Omega}\lambda^{d/2+1}-\frac{L_{d-1}}{4}\cH^{d-1}(\partial\Omega)\lambda^{(d-1)/2+1}+o(\lambda^{(d-1)/2+1}),$$
for constants $L_d,L_{d-1}\ge 0$.

\color{black}
We intend to prove a similar asymptotic formula for the hyperbolic Laplacian $H$.
\begin{lm} Assume that $\Omega\subset\IR^d$ is open and bounded and that $\partial\Omega$ is Lipschitz-regular. Then the trace of the hyperbolic Laplacian on $\Omega$ with Dirichlet boundary conditions satisfies
$$\sum_k (\lambda-{\lambda}_k)_+\ge \lambda^{1+d/2}\int_{\IR^d\times \Omega}\left(\lambda - \xi_1^2-e^{2y_1}\abs{\tilde{\xi}}^2\right)_+\frac{d\xi dy}{(2\pi)^d}+O(\lambda^{d/2+2/3}).$$
\end{lm}
\begin{proof}
Recall that in \ref{low} we found
$$\sum_k (\lambda-{\lambda}_k)_+\ge \int_{\IR^d\times {\Omega_{2\epsilon}}}\left(\lambda - \xi_1^2-e^{2y_1}\abs{\tilde{\xi}}^2c^\epsilon_3-e^{2y_1}c^\epsilon_2-c^\epsilon_1 \right)_+\frac{d\xi dy}{(2\pi)^d}.$$
Defining $B_\epsilon=\{(\xi, y)\in \IR^d\times \Omega_{2\epsilon}\colon \xi_1^2+e^{2y_1}\abs{\tilde{\xi}}^2\le 1\}$, one evidently has
\begin{align*}
\sum_k (\lambda-{\lambda}_k)_+&\ge \int_{B_\epsilon}\left(\lambda - \xi_1^2-e^{2y_1}\abs{\tilde{\xi}}^2c^\epsilon_3-e^{2y_1}c^\epsilon_2-c^\epsilon_1 \right)_+\frac{d\xi dy}{(2\pi)^d}\\
&=\lambda^{1+d/2} \int_{B_\epsilon}\left(1- \xi_1^2-e^{2y_1}\abs{\tilde{\xi}}^2c^\epsilon_3-\frac{1}{\lambda}e^{2y_1}c^\epsilon_2-\frac{1}{\lambda}c^\epsilon_1 \right)_+\frac{d\xi dy}{(2\pi)^d}.
\end{align*}

For the integrand we find
$$\left(1- \xi_1^2-e^{2y_1}\abs{\tilde{\xi}}^2c^\epsilon_3-\frac{1}{\lambda}e^{2y_1}c^\epsilon_2-\frac{1}{\lambda}c^\epsilon_1 \right)_+\ge \left(1- \xi_1^2-e^{2y_1}\abs{\tilde{\xi}}^2\right)_+-e^{2y_1}\abs{\tilde{\xi}}^2(c^\epsilon_3-1)_+-\frac{1}{\lambda}e^{2y_1}c^\epsilon_2-\frac{1}{\lambda}c^\epsilon_1 .$$
Consequently,
\begin{align*}
\sum_k (\lambda-{\lambda}_k)_+&\ge \lambda^{1+d/2}\int_{B} \left(1- \xi_1^2-e^{2y_1}\abs{\tilde{\xi}}^2\right)_+ \frac{d\xi dy}{(2\pi)^d}-\lambda^{1+d/2}\int_{B\setminus B_\epsilon} \left(1- \xi_1^2-e^{2y_1}\abs{\tilde{\xi}}^2\right)_+ \frac{d\xi dy}{(2\pi)^d}\\
& -\lambda^{1+d/2}(c_3^\epsilon-1)_+\frac{\abs{B}}{(2\pi)^d}-\lambda^{d/2}c_2^\epsilon e^{2R}\frac{\abs{B}}{(2\pi)^d}-\lambda^{d/2}c_1^\epsilon \frac{\abs{B}}{(2\pi)^d},
\end{align*}
where $B=\{(\xi, y)\in \IR^d\times \Omega\colon \xi_1^2+e^{2y_1}\abs{\tilde{\xi}}^2\le 1\}$ and $R=\sup\{y_1\colon\ y\in\Omega\}$. Note that
$$\int_{B\setminus B_\epsilon} \left(1- \xi_1^2-e^{2y_1}\abs{\tilde{\xi}}^2\right)_+ \frac{d\xi dy}{(2\pi)^d}\le \frac{\abs{B\setminus B_\epsilon}}{(2\pi)^d},$$
and that
\begin{align*}
B\setminus B_\epsilon&=\{(\xi,y)\in\IR^d\times(\Omega\setminus\Omega_{2\epsilon})\colon \xi_1^2+e^{2y_1}\abs{\tilde{\xi}}^2\le 1\}\\
&\subset\{(\xi,y)\in\IR^d\times(\Omega\setminus\Omega_{2\epsilon})\colon \xi_1^2+e^{2r}\abs{\tilde{\xi}}^2\le 1\}\\
&=\{\xi\in\IR^d\colon \xi_1^2+e^{2r}\abs{\tilde{\xi}}^2\le 1\}\times (\Omega\setminus\Omega_{2\epsilon}),
\end{align*}
where $r=\inf\{y_1\colon\ y\in\Omega\}$.
Clearly, the set $\{\xi\in\IR^d\colon \xi_1^2+e^{2r}\abs{\tilde{\xi}}^2\le 1\}$ is bounded and thus of finite measure. Therefore, assuming that $\partial\Omega$ is Lipschitz-regular, we can infer that
$$\frac{\abs{B\setminus B_\epsilon}}{(2\pi)^d}= O(\epsilon),\ \epsilon\to 0.$$
Conside now the term $c_3^\epsilon-1$. We know that
$$
c^\epsilon_3=\int_{\IR^d}e^{2z_1}g^\epsilon(z)^2dz=\int_{\IR}e^{2z_1}g_1^\epsilon(z_1)^2dz_1=\int_{\IR}e^{2z_1\epsilon}g_1(z_1)^2dz_1.
$$
Hence 
$$c_3^\epsilon -1 = \int_{\IR}(e^{2z_1\epsilon}-1)g_1(z_1)^2dz_1=\int_{\IR}\sum_{k=1}^\infty \frac{(2z_1)^k}{k!}\epsilon^k g_1(z_1)^2dz_1=O(\epsilon),\ \epsilon\to 0.$$
For $c_2^\epsilon$ we find
$$c^\epsilon_2=\int_{\IR^d}e^{2z_1}\abs{{\nabla_{\tilde{z}}}g^\epsilon(z)}^2dz=\epsilon^{-2}\int_{\IR^d}e^{2z_1\epsilon}g(z_1)^2dz_1\int_{\IR^d}e^{2z_1}\abs{{\nabla_{\tilde{z}}}\tilde{g}(\tilde{z})}^2d\tilde{z}=O(\epsilon^{-2}),\ \epsilon\to 0,$$
and finally
$$c^\epsilon_1=\int_{\IR^d}(\partial_{z_1}g^\epsilon(z))^2 dz=\int_{\IR}(\partial_{z_1}g_1^\epsilon(z_1))^2 dz_1\int_{\IR^{d-1}}\tilde{g}(\tilde{z})^2 d\tilde{z}=\epsilon^{-2}\int_{\IR}(\partial_{z_1}g_1(z_1))^2 dz_1=O(\epsilon^{-2}).$$
Linking $\epsilon$ to $\lambda$, again, via $\epsilon=\lambda^{-\alpha}$ we obtain
$$\sum_k (\lambda-{\lambda}_k)_+\ge \lambda^{1+d/2}\int_{\IR^d\times \Omega}\left(\lambda - \xi_1^2-e^{2y_1}\abs{\tilde{\xi}}^2\right)_+\frac{d\xi dy}{(2\pi)^d}-R(\lambda),$$
where $R(\lambda)$ is a linear combination of terms that are in $O(\lambda^{d/2+1-\alpha})$ and $O(\lambda^{d/2-2\alpha})$ respectively. Again, minimizing $\max\{d/2+2\alpha,\ 1+d/2-\alpha\}$ yields $\alpha=1/3$ and thus
$$\sum_k (\lambda-{\lambda}_k)_+\ge \lambda^{1+d/2}\int_{\IR^d\times \Omega}\left(\lambda - \xi_1^2-e^{2y_1}\abs{\tilde{\xi}}^2\right)_+\frac{d\xi dy}{(2\pi)^d}+O(\lambda^{d/2+2/3}).$$
\end{proof}
In contrast to the case of the euclidean Laplacian, we need to prove an asymptotic upper bound for the hyperbolic Laplacian as well.
\begin{lm}Let the same conditions as in the preceding lemma be satisfied. Then
$$\sum_k (\lambda-{\lambda}_k)_+\le\lambda^{1+d/2} \int_{\IR^{d}\times \Omega}(1- {\xi_1}^2-e^{2y_1}\abs{\tilde{\xi}}^2)_+ \frac{d\xi dy}{(2\pi)^d}+O(\lambda^{d/2+2/3})$$
\end{lm}
\begin{proof}
Recall that 
$$\sum_k(\lambda-\lambda_k)_+\le\lambda^{1+d/2}\int_{\IR^{d}\times \Omega^\epsilon}(1-\xi_1^2-(c^\epsilon_3)^{-1}e^{2y_1}\abs{\tilde{\xi}}^2+\lambda^{-1}c^\epsilon_1+\lambda^{-1}(c^\epsilon_3)^{-1}c^\epsilon_2e^{2R})_+ \frac{d\xi dy}{(2\pi)^d}.$$
Let $\epsilon=\lambda^{-\alpha}$. Then $((1-(c^\epsilon_3)^{-1})=O(\lambda^{-\alpha})$, $\lambda^{-1}c_1^\epsilon=O(\lambda^{-1+2\alpha})$ and $\lambda^{-1}c_2^\epsilon=O(\lambda^{-1+2\alpha})$. If $\alpha<1/2$, then all these terms converge to $0$. In this case, for sufficiently large $\lambda$, the integrand is supported in 
$B^\epsilon=\{(\xi, y)\in \IR^d\times \Omega^\epsilon\colon \xi_1^2+e^{2y_1}\abs{\tilde{\xi}}^2\le 2\}$.
Furthermore, the integrand can be estimated from above in the following way
\begin{align*}
(1-\xi_1^2-(c_3^\epsilon)^{-1}e^{2y_1}\abs{\tilde{\xi}}^2+c^\epsilon_1\lambda^{-1}+(c^\epsilon_3)^{-1}c^\epsilon_2e^{2R}\lambda^{-1})_+&\le (1-\xi_1^2-e^{2y_1}\abs{\tilde{\xi}}^2)_+ +(1-(c^\epsilon_3)^{-1})_+e^{2R}\abs{\tilde{\xi}}^2\\
&+\lambda^{-1}c^\epsilon_1+\lambda^{-1}(c^\epsilon_3)^{-1}c^\epsilon_2e^{2R}.
\end{align*}
Define $B=\{(\xi, y)\in \IR^d\times \Omega\colon \xi_1^2+e^{2y_1}\abs{\tilde{\xi}}^2\le 2\}$. Then
\begin{align*}
\sum_k(\lambda-\lambda_k)_+&\le \lambda^{1+d/2} \int_{B}(1- \xi_1^2-e^{2y_1}\abs{\tilde{\xi}}^2)_+ \frac{d\xi dy}{(2\pi)^d}\\
&+\lambda^{1+d/2} \int_{B^\epsilon\setminus B}(1- \xi_1^2-e^{2y_1}\abs{\tilde{\xi}}^2)_+ \frac{d\xi dy}{(2\pi)^d}\\
&+2e^R\lambda^{d/2+1}((1-(c^\epsilon_3)^{-1})_+\frac{\abs{B^\epsilon}}{(2\pi)^d}
+\lambda^{d/2}c^\epsilon_1\frac{\abs{B^\epsilon}}{(2\pi)^d}
+\lambda^{d/2}(c^\epsilon_3)^{-1}c^\epsilon_2e^{2R}\frac{\abs{B^\epsilon}}{(2\pi)^d}.
\end{align*}
Note that $$\int_{B^\epsilon\setminus B}(1- {\xi_1}^2-e^{2y_1}\abs{\tilde{\xi}}^2)_+ \frac{d\xi dy}{(2\pi)^d}\le \frac{\abs{B^\epsilon\setminus B}}{(2\pi)^d}=O(\epsilon)=O(\lambda^{-\alpha}).$$
Therefore, the error term is of the order $O(\lambda^{\max\{d/2+2\alpha,\ 1+d/2-\alpha\}})$. As for the lower bound, the exponent is minimized for $\alpha=1/3$. This concludes the argument.
\end{proof}
Combining the two preceding lemmas, we obtain
\begin{theorem}Let $\Omega\subset\IR^d$ be a bounded open set with Lipschitz boundary and let $\lambda_k$, $k\in\IN$ be the eigenvalues of the operator $H$ on $\Omega$ with Dirichlet boundary conditions. Then
$$\sum_k (\lambda-{\lambda}_k)_+=\lambda^{1+d/2} \int_{\IR^{d}\times \Omega}(1- {\xi_1}^2-e^{2y_1}\abs{\tilde{\xi}}^2)_+ \frac{d\xi dy}{(2\pi)^d}+O(\lambda^{d/2+2/3}).$$
\end{theorem}
\section*{Acknoledgments}
I would like to thank Ari Laptev for his help in preparing this paper. Furthermore I am indebted to the reviewers for their extremely valuable comments.
\bibliographystyle{alpha}
\bibliography{main} 
\end{document}